\documentclass{article}
\usepackage[style=numeric,maxnames=99]{biblatex}
\usepackage{hyperref}
\usepackage{amsmath}
\usepackage{tikz}
\usepackage{tikz-cd}
\usepackage{amsthm}
\usepackage{caption}
\usepackage{ amssymb }

\title{There is No Composition in the Computable Reducibility Degrees}
\author{Daniel Mourad}
\theoremstyle{plain}
\newtheorem{Theorem}{Theorem}[section]
\newtheorem{Proposition}[Theorem]{Proposition}

\theoremstyle{definition}
\newtheorem{Definition}[Theorem]{Definition}

\theoremstyle{plain}

\DeclareMathOperator{\dom}{dom}

\DeclareMathOperator{\id}{id}
 
\newcommand\wdoublehat[2][0mu]{\mskip#1\widehat{\mskip-#1\widehat{#2}\mskip#1}\mskip-#1\relax}
\newcommand{\Xdoublehat}{\wdoublehat[1.5mu]{X}}

\begin{document}
	\maketitle
	\begin{abstract}
		We show that, in general, there is no degree corresponding to the composition of two problems in the computable reducibility lattice. 
	\end{abstract}
	\section{Introduction}
	In the program of reverse mathematics, single use reductions between instance-solution problems provide a way to analyze observations such as ``Theorem $P$ can be proven using one application of Theorem $Q$''.
	
	The compositional product $\star$ in the Weihrauch degrees extends this functionality. Using the $\star$ operator, one can also substantiate statements such as ``$P$ can be proven with exactly two applications of $Q$ and no fewer'' and ``Theorem $R$ can be proven with an application of $Q$ followed by an application of $P$, but not the other way around''. In particular, the $\star$ operator has also been used to break proofs into pieces, as well as to put a bound on the degree of non-uniformity in computable reductions (see, \emph{e.g.}, \cite{SoldaShaferFioricarones2022}). 
	
	However, despite the success of the compositional product for Weihrauch reducibility, a compositional product has not been discovered in the computable reducibility degrees. The main result of this paper is that this is by necessity. 
	
	We consider a notion of composition whose analogue under Weihrauch reducibility characterizes the $\star$-product. 
	
	\theoremstyle{definition}
	\newtheorem*{def:composition}{Definition \ref{SetDefinitionOfComposition}}
	\theoremstyle{plain}
	\begin{def:composition}
		Let $P$ and $Q$ be problems. Define $P *_c Q$ to be the set of compositions
		\[
		P *_c Q =  \{P' \circ Q':  P' \leq_c P \text{ and } Q' \leq_c Q\}.
		\]  
		
		\noindent
		For problem $R$, we say that $R \leq_c P *_c Q$ if and only if there is a $(P' \circ Q') \in (P *_c Q)$ such that $R \leq_c P' \circ Q'$.
	\end{def:composition}
	
	We show that there is no computable reducibility degree corresponding with this notion of reducibility. 
	
	\newtheorem*{thm:nocomposition}{Theorem \ref{NoCompositionInComputableReductionDegrees}}
	\begin{thm:nocomposition}[No Composition Theorem]
		There are problems $P$ and $Q$ such that for all problems $S$ there exists a problem $R$ such that $R \leq_c P *_c Q$ but $R \not\leq_c S$. 
	\end{thm:nocomposition}
	
	To support the claim that Theorem \ref{NoCompositionInComputableReductionDegrees} shows that there is no compositional product in the computable reducibility degrees, we introduce another notion of computable reducibility to a composition of two problems in Definition \ref{QFollowedByP} and show that it is equivalent to the notion stated in Definition \ref{SetDefinitionOfComposition}. 
	
	\section{Notions}
	
	We work with reductions on instance solution problems. 
	
	\begin{Definition}
		A problem $P: \subseteq \omega^\omega \rightrightarrows \omega^\omega$ is a partial multi-function taking $\omega$-valued sequences to sets of $\omega$-valued sequences. We call $\dom(P)$ the \emph{instances} of $P$, or the set of $P$\emph{-instances}. For instance $X$ of $P$, we say that elements of $P(X)$ are the $P$\emph{-solutions} of $X$. 
	\end{Definition}

	We now define computable reducibility of two problems, with the forwards function specified. 
	
	\begin{Definition}
		\label{ComputablyReducibleViaF}
		Let $P$ and $Q$ be problems. We say that $P \leq_c Q$ via function $F: \dom(P) \to \dom(Q)$ if for all $X \in \dom(P)$ we have that $X$ computes $Q$-instance $\widehat{X} := F(X) \leq_T X$ such that for each $\widehat{Y} \in Q(\widehat{X})$ there is a $Y \leq_T X \oplus \widehat{Y}$ such that $Y \in P(X)$ (see Figure \ref{fig:compReducibility}). 
		
		\noindent In this case, we call $F$ the \emph{forward function} of the computable reduction of $P$ to $Q$. 
	\end{Definition}

	Note that $P \leq_c Q$ if and only if there is $F$ such that $P \leq_c Q$ via $F$, with the implication going from left to right using the axiom of choice. 

		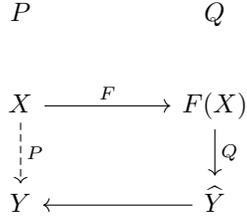
\begin{figure}
			\begin{center}
			\begin{tikzcd}
				P & & Q\\
				X \arrow[rr,"F"] \arrow[d, dashed,"P"] & & F(X) \arrow[d, "Q"]\\
				Y & & \widehat{Y} \arrow[ll] 
			\end{tikzcd}
		\end{center}
		\caption{Computable reducibility via $F$}
		\label{fig:compReducibility}
		\end{figure}
	
	\citeauthor{Dzhafarov2022} \cite{Dzhafarov2022} give a notion of computable reducibility to multiple applications of a problem. In the same vein, we consider computable reducibility of a problem relative to using a problem followed by another problem. Again, we also give a definition with the forwards functions specified.
	
	\begin{Definition}
		\label{QFollowedByP}
		Let $P$, $Q$, and $R$ be problems. We say that $R$ is computably reducible to $Q$ followed by $P$ if 
		\begin{itemize} \item for all $R$-instances $X$ there is $Q$-instance $\widehat{X} \leq_T X$ such that 
			\item for all $Q$-solutions $\widehat{Y}$ of $\widehat{X}$ there is $P$-instance $\Xdoublehat \leq_T X \oplus \widehat{Y}$ such that
			\item for all $P$-solutions $\widehat{\widehat{Y}}$ of $\Xdoublehat$, there is $R$-solution $Y \leq_T X \oplus \widehat{Y} \oplus \widehat{\widehat{Y}}$ of $X$. 
		\end{itemize}
		We say that $R$ is computably reducible to $Q$ followed by $P$ via functions $F$ and $G$ if $F(X) = \widehat{X}$ and $G(X,\widehat{Y}) = \Xdoublehat$.
	\end{Definition}
	We once again have that, using the axiom of choice, if $R$ is reducible to $Q$ followed by $P$ then there are $F$ and $G$ such that $R$ is computably reducible to $Q$ followed by $P$ via $F$ and $G$.
		
	Definition \ref{QFollowedByP} is in the spirit of the reduction games of \citeauthor{Hirschfeldt2016} \cite{Hirschfeldt2016} and \citeauthor{Goh2020} \cite{Goh2020}. $R$ being computably reducible to $Q$ followed by $P$ is equivalent to player 2  having a winning strategy in the $G\left([Q \sqcup P] \rightarrow R\right)$ game \cite{Hirschfeldt2016} in which the first move always uses $Q$, the second move always uses $P$, and the third move solves $R$. It is also the computable reducibility analogue of $R \leq_W^{(2)} Q,P$ \cite{Goh2020}.

	\citeauthor{MarconeGherardiBrattka2012} \cite{MarconeGherardiBrattka2012} studied composition in the Weihrauch degrees as $P * Q = \sup_{\leq_W}\{P' \circ Q': P' \leq_W P \text{ and } Q' \leq_W Q\}$, where $\circ$ denotes set-theoretic composition of partial multi-functions. \citeauthor{Brattka2018} \cite{Brattka2018} showed that this supremum always exists and defined the problem $P \star Q$ as a representative of $P * Q$. 
	
	\begin{Definition}
		Let $P$ and $Q$ be problems. Define their composition $P \circ Q$ to be the problem whose instances are $X \in \dom(Q)$ such that for all $Y \in Q(X)$ we have that $Y \in \dom(P)$. $(P \circ Q)$-solutions of $X$ are any $Z$ such that there exists $Y \in Q(X)$ such that $Z \in P(Y)$. 
	\end{Definition}

	We state the analogous definition of $P * Q$ for computable reducibility. 
	
	\begin{Definition}
		\label{SetDefinitionOfComposition}
		Let $P$ and $Q$ be problems. Define $P *_c Q$ to be the set of compositions
		\[
		P *_c Q =  \{P' \circ Q':  P' \leq_c P \text{ and } Q' \leq_c Q\}.
		\] 
		%where $P' \circ Q'$ is only considered for compatible $P'$ and $Q'$. 
		
		\noindent
		For problem $R$, we say that $R \leq_c P *_c Q$ if and only if there is a $(P' \circ Q') \in (P *_c Q)$ such that $R \leq_c P' \circ Q'$. 
	\end{Definition}
	
	Together, Proposition \ref{DefinitionsCoincide} and Theorem \ref{NoCompositionInComputableReductionDegrees} show that, unlike the analogous case for Weihrauch degrees, the supremum of $P *_c Q$ does not exist in general.

	\section{Results}

	First, we show that Definitions \ref{QFollowedByP} and \ref{SetDefinitionOfComposition} coincide. 
	
	\begin{Proposition}
		\label{DefinitionsCoincide}
		Let $P$, $Q$, and $R$ be problems. Then, $R$ is computably reducible to $Q$ followed by $P$ if and only if $R \leq_c P *_c Q$. 
	\end{Proposition}
	\begin{proof}
		First suppose that $R$ is computably reducible to $Q$ followed by $P$ via $F$ and $G$. Let $P' = \id \times P$. Let $Q'$ be the problem whose instances are $(X,F(X))$ such that $X \in \dom(R)$ and whose solutions to $(X,F(X))$ are $(\widehat{Y},G(X,\widehat{Y}))$ where $\widehat{Y} \in Q(F(X))$.  To see that $R \leq_c Q' \circ P'$, let $X$ be an instance of $R$. Then, $X \geq_T (X,F(X))$ is an instance of $P'$. Then, for any $(\widehat{Y},\widehat{\widehat{Y}}) \in (P' \circ Q')(X,F(X))$, we have that setting  $\widehat{X} := F(X)$, $\widehat{Y} := \widehat{Y}$, $\Xdoublehat = G(X,\widehat{Y})$, and $\widehat{\widehat{Y}} := \widehat{\widehat{Y}}$ witnesses the definition of $R$ being computable reducible to $Q$ followed by $P$ and hence that there is $Y \in R(X)$ such that $Y \leq_c X \oplus \widehat{Y} \oplus \widehat{\widehat{Y}}$, completing the reduction $R \leq_c P' \circ Q'$. 
		
		Now suppose that $R \leq_c P *_c Q$. Then, there is $(P' \circ Q') \in (P *_c Q)$ such that $R \leq_c P' \circ Q'$, say via $H$. Also suppose that $Q' \leq_c Q$ via $F$ and $P' \leq_c P$ via $G$. To show that $R$ is computably reducible to $Q$ followed by $P$, send $X \in \dom(R)$ to $\widehat{X} := G(H(X)) \in \dom(Q)$. Then, for any $\widehat{Y} \in Q(\widehat{X})$, there is $\widehat{Y}' \in Q'(H(X))$ such that $\widehat{Y}' \leq_T X \oplus \widehat{Y}$. Since $Q'$ and $P'$ are compatible, we have that $\widehat{Y}' \in \dom(P')$, so $\Xdoublehat := G(\widehat{Y}{'}) \in \dom(P)$. Then, from any $\widehat{\widehat{Y}} \in P(\Xdoublehat)$, we can compute $\widehat{\widehat{Y}}{'} \in P'(\widehat{Y}{'})$. Since $\widehat{\widehat{Y}}{'}$ is a solution to $(P' \circ Q')(H(X))$ there is $Y \in R(X)$ such that $\widehat{\widehat{Y}}{'} \geq_T Y$, completing the reduction of $R$ to $Q$ followed by $P$. 
	\end{proof}
	
	In the Weihrauch degrees, $R \leq_W P \star Q$ if there is a uniform way to obtain instances of $R$ into instances of $P$ with a uniform way of obtaining instances of $Q$ from the $P$-solution. However, a computable reduction is able to use \textit{any} instance of $P$ computable from a $Q$-solution. The computable reducibility degrees are unable to capture this flexibility, in part because, while there are only countably many ways to use $Q$ followed by $P$ in a uniform way, there are uncountably many ways to use $Q$ followed by $P$ in a non-uniform way. 
	
	\begin{Theorem}[No Composition Theorem]
		\label{NoCompositionInComputableReductionDegrees}
		There are problems $P$ and $Q$ such that there is no problem $S$ such that for all problems $R$ it is the case that $S \geq_c R$ if and only if $R$ is computably reducible to $Q$ followed by $P$.
	\end{Theorem}
	\begin{proof}
		Let $\mathcal{A}$ be an infinite collection of Turing independent sets and let $A_0$ and $A_1$ be such that $\mathcal{A} \cup \{A_0,A_1\}$ is also Turing independent. Let $Q$ be the problem whose instances are computable sets and whose solutions to $X$ are any element of $\mathcal{A}$. Let $P$ be the single valued problem with two instances, $0$ and $1$, such that $A_0$ is the only $P$-solution to $0$ and $A_1$ is the only $P$-solution to $1$. 
		
		We will show that there are too many problems which are reducible to $Q$ followed by $P$. For each function $f: \mathcal{A} \to 2$, let $R_f$ be the problem whose instances are computable sets and whose solutions are any $(A,A_i)$ such that $A \in \mathcal{A}$ and $f(A) = i$. Then we have that $R_f$ is computably reducible to $Q$ followed by $P$ for each $f$. To see this, let $X$ be an instance of $R_f$. Then, $X$ is also an instance of $Q$. Let $A \in Q(X) = \mathcal{A}$. Then, $f(A)$ is an instance of $P$ such that $P(f(A)) = A_{f(A)}$. The pair $(A,A_{f(A)})$ is computable from $A \oplus A_{f(A)}$, completing the reduction. 
		
		Now suppose for contradiction there is problem $S$ such that $R \leq_c S$ if and only if $R$ is computably reducible to $Q$ followed by $P$. Then, in particular, we have that $R_f \leq_c S$ for each $f: \mathcal{A}\to 2$ and that $S$ is computably reducible to $Q$ followed by $P$. 
		
		Let $2^\mathcal{A}$ be the set of functions from $\mathcal{A}$ to 2. Fix computable set $X$. So that the notation matches Definition \ref{QFollowedByP}, we will superfluously write computable sets in Turing reductions. Due to the reductions $R_f \leq_c S$, we have that for each $f \in 2^\mathcal{A}$, there is $\widehat{X} \leq_T X$ such that $\widehat{X} \in \dom(S)$ and such that for each $Y \in S(\widehat{X})$ it is the case that $Y \geq_T (A,A_{f(A)})$ for some $A \in \mathcal{A}$. Because ${2^\mathcal{A}}$ is uncountable and $S$ can only have up to countably many computable instances, by the pigeonhole principle there must be $f_0 \neq f_1 \in {2^\mathcal{A}}$ which share such an $\widehat{X}$. Fix such an $f_0,f_1$, and $\widehat{X}$. 
		
		Fix $A \in \mathcal{A}$ such that $f_0(A) \neq f_1(A)$. Since $S$ is computably reducible to $Q$ followed by $P$, there is $Y \in S(\widehat{X})$ such that 
		$Y \leq_T X \oplus A \oplus A_i$ for some $i \in \{0,1\}$. By our choice of $\widehat{X}$, we also have that there are $A'$ and $A''$ such that $Y \geq_T (A',A_{f_0(A')})$ and $Y \geq_T (A'',A_{f_1(A'')})$. By transitivity of $\leq_T$, we have that $X \oplus A \oplus A_{i} \geq_T  (A',A_{f_0(A')})$ and $X \oplus A \oplus A_{i} \geq_T (A'',A_{f_1(A'')})$. 	
		By Turing independence of $\mathcal{A} \cup \{A_0,A_1\}$, we have that $A = A' = A''$ and $i = f_0(A') = 
		f_1(A'')$. However, this contradicts that $f_0(A) \neq f_1(A)$. 
		%Since $R_f \leq_c S$, we know that $S$ has computable instances. Because $S$ is reducible to $Q$ followed by $P$, we know that for each $A \in \mathcal{A}$, each computable instance of $S$ has a solution $Y$ such that either $Y \leq_T A \oplus A_0$ or $Y \leq_T A \oplus A_1$.
		%Consider the reductions from 
		%For each $f: \mathcal{A} \to 2$, the reduction $R_f \leq_c S$ tells us that for each $R_f$, there is a computable $X \in \dom(S)$ such that for each $Y \in S(X)$ we have that $(Z_Y,A_{f(Z_Y)}) \leq_T Y$ for some $Z_Y \in \mathcal{A}$. By the previous paragraph, we have that $X$ has a solution $Y'$ such that either $Y' \leq_T Z_Y \oplus A_0$ or $Y' \leq_T  Z_Y \oplus A_1$. 
	\end{proof}
	
	We conclude by noting that one can extend Theorem \ref{NoCompositionInComputableReductionDegrees} to applications of the same problem. For the same $Q$ and $P$, a similar argument shows that there is no representative in the $\leq_c$-degrees of computable reducibility to two applications of $Q \sqcup P$.

	The author thanks Damir Dzhafarov for his feedback on early versions of these results.

\newpage
\printbibliography
\end{document}